\documentclass[11pt, reqno]{amsart}

\usepackage{amssymb,dsfont}
\usepackage{hyperref}
\usepackage[matrix, arrow, curve]{xy}
\usepackage{dsfont}
\usepackage{stmaryrd}
\usepackage{mathabx}

\usepackage[english]{babel}
\usepackage{ulem}

\theoremstyle{plain}
\newtheorem{theorem}{Theorem}[section]

\newtheorem{proposition}[theorem]{Proposition}
\newtheorem{corollary}[theorem]{Corollary}

\theoremstyle{remark}

\theoremstyle{definition}
\newtheorem{example}[theorem]{Example}

\newtheorem{definition}[theorem]{Definition}

\numberwithin{equation}{section}

\begin{document}

\title[Universal coacting Poisson Hopf algebras]{Universal coacting Poisson Hopf algebras}

\author{A.L. Agore}
\address{Simion Stoilow Institute of Mathematics of the Romanian Academy, P.O. Box 1-764, 014700 Bucharest, Romania}
\address{Vrije Universiteit Brussel, Pleinlaan 2, B-1050 Brussels, Belgium}
\email{ana.agore@vub.be,\, ana.agore@gmail.com}


\begin{abstract} We introduce the analogue of Manin's universal coacting (bialgebra) Hopf algebra for Poisson algebras. First, for two given Poisson algebras $P$ and $U$, where $U$ is finite dimensional, we construct a Poisson algebra $\mathcal{B}(P,\, U)$ together with a Poisson algebra homomorphism $\psi_{\mathcal{B}(P,\,U)} \colon P \to U \otimes \mathcal{B}(P,\, U)$ satisfying a suitable universal property. $\mathcal{B}(P,\, U)$ is shown to admit a Poisson bialgebra structure for any pair of Poisson algebra homomorphisms subject to certain compatibility conditions. If $P=U$ is a finite dimensional Poisson algebra then $\mathcal{B}(P) = \mathcal{B}(P,\, P)$ admits a unique Poisson bialgebra structure such that $\psi_{\mathcal{B}(P)}$ becomes a Poisson comodule algebra and, moreover, the pair $\bigl(\mathcal{B}(P),\, \psi_{\mathcal{B}(P)}\bigl)$ is the universal coacting bialgebra of $P$. The universal coacting Poisson Hopf algebra $\mathcal{H}(P)$ on $P$ is constructed as the initial object in the category of Poisson comodule algebra structures on $P$ by using the free Poisson Hopf algebra on a Poisson bialgebra (\cite{A1}). 
\end{abstract}


\thanks{This research was supported by a grant of Romanian Ministery of Research and Innovation, CNCS - UEFISCDI, project number PN-III-P1-1.1-TE-2016-0124, within PNCDI III. The author is a fellow of FWO (Fonds voor Wetenschappelijk Onderzoek -- Flanders).}

\maketitle

\section{Introduction}

Poisson structures appear naturally in various branches of mathematics and mathematical physics ranging from algebra and non-commutative geometry to classical and quantum mechanics. A Poisson Hopf algebra is both a Poisson algebra and a Hopf algebra sharing the same (commutative) algebra structure and such that the comultiplication and the counit are Poisson algebra homomorphisms. As one of their many important applications we mention the quantization theory for Lie bialgebras. Since they were first considered by Drinfel'd (\cite{Dri}) more than 30 years ago, Poisson Hopf algebra structures appeared gradually in the study of quantum groups, homological algebra, Poisson geometry or representation theory. Among the naturally occurring examples of such objects is the algebra of smooth functions on a Poisson group. Another important example of a Poisson Hopf algebra which recently surfaced (see \cite{zhuang}) is the associated graded algebra with respect to the coradical filtration of a connected Hopf algebra. This motivates the idea of considering connected Hopf algebras as some sort of deformations of Poisson Hopf algebras. Therefore, studying (co)representations of Poisson Hopf algebras will lead to a better understanding of (co)representations of Hopf algebras in general. 

In the present note we introduce the universal coacting Poisson Hopf algebra of a finite dimensional Poisson algebra. More precisely, given a Poisson algebra $P$ we construct a Poisson Hopf algebra $\mathcal{H}(P)$ together with a right Poisson $\mathcal{H}(P)$-comodule algebra structure $\psi_{\mathcal{H}(P)}$ on $P$ which is an initial object in the category of Poisson comodule algebra structures on $P$ (see Definition~\ref{00} and Definition~\ref{01}).

This construction can be seen as the Poisson Hopf algebra counterpart of the universal coacting bialgebra/Hopf algebra of a (graded) algebra $A$ which appeared almost  simultaneously in the work of Yu.I. Manin (\cite{Manin}) and D. Tambara (\cite{Tambara}). In some sense Manin's universal coacting Hopf algebra captures the non-commutative symmetries of the algebra $A$ and it plays the role of a symmetry group in non-commutative geometry.

Aside from being of interest in his own right, the universal coacting Poisson Hopf algebra might have some relevance in physics as well. To be more precise, integrable systems have been recently constructed by using  Poisson comodule algebras in \cite{BMR}. Furthermore, comodule algebra symmetry was shown to be applicable to the construction of new integrable deformations of certain Smorodinsky-Winternitz systems (\cite{BHO}).

An outline of the paper is as follows. In Section~\ref{pre} we provide the necessary background on Poisson (Hopf) algebras together with some category theory results which will be used in the sequel. Section~\ref{univ_coact} contains the main results. Given two Poisson algebras $P$ and $U$, where $U$ is finite dimensional, the universal Poisson algebra $\bigl(\mathcal{B}(P,\,U),\, \psi_{\mathcal{B}(P,\, U)}\bigl)$ of $P$ and $U$ is constructed in Theorem~\ref{1}. Proposition~\ref{pro} shows that the universal Poisson algebra gives rise to two functors $\mathcal{L}(-,\, U) \colon \mathbf{Poiss}_{F} \to \mathbf{Poiss}_F$ and respectively $\mathcal{R}(P,\, -) \colon \bigl(\mathbf{Poiss}_{F}^{fd}\bigr)^{op} \to \mathbf{Poiss}_F$. Furthermore, the functor $\mathcal{L}(-,\,U)$ is left adjoint to the tensor functor $U \otimes -  \colon \mathbf{Poiss}_{F} \to \mathbf{Poiss}_F$ (Theorem~\ref{th11}). It turns out that any Poisson algebra homomorphism $f \colon U \to P$ gives rise to a coassociative Poisson algebra homomorphism $\Delta_{f} \colon \mathcal{B}(P,\, U) \to \mathcal{B}(P,\, U) \otimes \mathcal{B}(P,\, U)$ and, moreover, if $g \colon P \to U$ is another Poisson algebra homomorphism compatible with $f$ in a certain sense (see Theorem~\ref{imp1}) then $\mathcal{B}(P,\, U)$ becomes a Poisson bialgebra whose counit $\varepsilon_{g} \colon \mathcal{B}(P,\, U) \to F$ is induced by $g$. When $P = U$ is a finite dimensional Poisson algebra the universal Poisson algebra $\mathcal{B}(P, P)$, denoted simply by $\mathcal{B}(P)$, admits a unique Poisson bialgebra structure which makes $\psi_{\mathcal{B}(P)}$ into a Poisson comodule algebra. With this Poisson bialgebra structure, the pair $\bigl(\mathcal{B}(P),\, \psi_{\mathcal{B}(P)}\bigl)$ is proved to be the universal Poisson bialgebra of $P$ (see Definition~\ref{01}). Furthermore, the existence of a free Poisson Hopf algebra on any Poisson bialgebra (\cite{A1}) allows us to construct a universal coacting Poisson Hopf algebra $\bigl(\mathcal{H}(P),\, \psi_{\mathcal{H}(P)}\bigl)$ on any finite dimensional Poisson algebra $P$, i.e. $\psi_{\mathcal{H}(P)}$ is a Poisson comodule algebra structure on $P$ such that for any other Poisson Hopf algebra $H$ and any Poisson comodule algebra structure $\rho_{H} \colon P \to P \otimes H$ there exists a unique Poisson Hopf algebra homomorphism $g\colon \mathcal{H}(P) \to H$ for which the following diagram commutes:
$$
\xymatrix{ P\ar[rr]^-{\psi_{\mathcal{H}(P)}}\ar[rrd]_{\psi_{H}}  & {} & {P \otimes \mathcal{H}(P)}\ar[d]^{ \mathds{1}_{P} \otimes g}  \\
{} & {} & {P \otimes H}
}
$$

\section{Preliminaries}\label{pre}
Throughout this paper, $F$ denotes a field and unless otherwise specified, all vector spaces, unadorned tensor
products, homomorphisms, (co)algebras,  bialgebras, Lie algebras, Poisson algebras,
Hopf algebras and Poisson Hopf algebras are over $F$. All (co)algebras considered are (co)associative and (co)unital.  The multiplication and unit maps of an algebra $A$ are denoted by $m_{A}$ and $u_{A}$ (or simply by $1_{A}$). Similarly, we use $\Delta_{C}$ and $\varepsilon_{C}$ to designate the comultiplication and counit of a coalgebra $C$. Given a vector space $V$, we denote by $\mathds{1}_{V}$ the identity map on $V$ and by $\mu_{V}\colon V \to V \otimes F$ the linear isomorphism defined by $\mu_{V}(v) = v \otimes 1_{F}$ for all $v \in V$. All comodules considered in this paper will typically be right comodules.

Our notation for the standard categories is as follows:  ${}_{F}{\mathcal {M}}$
(vector spaces), $\mathbf{Coalg}_F$ (coalgebras), $\mathbf{Poiss}_F$ (Poisson algebras), $\mathbf{Poiss}^{fd}_{F}$ (finite dimensional Poisson algebras), $\mathbf{PoissBialg}_F$ (Poisson bialgebras), $\mathbf{PoissHopf}_F$ (Poisson Hopf algebras). Given a category $\mathcal{C}$ we denote by ${\rm Hom}_{\mathcal{C}}(C,\,D)$ the set of morphisms in $\mathcal{C}$ between the objects $C$ and $D$. 

Recall (from \cite[Definition 1.1]{LPV}, for example) that a Poisson algebra is both an associative
commutative algebra and a Lie algebra living on the same vector
space $P$ such that for all $p$, $q$, $r
\in P$ we have:
$$[p,\, qr] = [p,\,q]\,r + q\, [p,\, r]$$

A linear map $f: P_{1} \to P_{2}$ is called a morphism of
Poisson algebras if $f$ is both an algebra homomorphism as well as a Lie
algebra homomorphism. A Poisson ideal is a linear subspace which
is both an ideal with respect to the associative product as well
as a Lie ideal. If $\mathcal{I}$ is a Poisson ideal of $P$ then
the quotient $P/\mathcal{I}$ becomes a Poisson algebra with respect to the obvious algebra and Lie algebra structures. 

If $P_{1}$, $P_{2}$ are Poisson algebras then the tensor product $P_{1} \otimes
P_{2}$ has a Poisson algebra structure defined for all $p$, $r\in
P_1$ and $q$, $s\in P_2$ by:
\begin{equation}\label{1}
(p \otimes q) \cdot (r \otimes s) := pr \otimes qs, \qquad \left[p
\otimes q, \, r \otimes s\right]_{P_{1} \otimes P_{2}} := pr \otimes [q, \, s]_{P_{2}}
+ [p, \, r]_{P_{1}} \otimes qs 
\end{equation}
Furthermore, for any Poisson algebra $P$ we have the obvious tensor product functor $P \otimes - \colon \mathbf{Poiss}_{F} \to \mathbf{Poiss}_{F}$.

Note that if $(\mathfrak{g},\, [\,,\,])$ is a Lie algebra then the symmetric algebra $S(\mathfrak{g})$ carries a canonical Poisson algebra structure with Lie algebra structure $\{\,,\,\}$ induced by that of $\mathfrak{g}$, i.e. $\{g,\,h\} = [g,\,h]$ for all $g$, $h \in \mathfrak{g}$. In particular, if $\mathfrak{g} = \mathcal{F}(V)$ is the free Lie algebra on a vector space $V$ then the symmetric algebra $S\bigl(\mathcal{F}(V)\bigl)$ is called the free Poisson algebra on $V$ and we will denote it by $\mathcal{P}(V)$. In fact, the functor sending a vector space $V$ to the free Poisson algebra $\mathcal{P}(V)$ provides a left adjoint for the forgetful functor $\mathbf{Poiss}_F \to {}_{F}{\mathcal {M}}$, i.e. there exists a linear homomorphism $\mu_{V} \colon V \to \mathcal{P}(V)$ such that for every Poisson algebra $P$ and any linear homomorphism $f \colon V \to P$ there exists a unique Poisson algebra homomorphism $g \colon \mathcal{P}(V) \to P$ such that the following diagram is commutative:
\begin{eqnarray*}
\xymatrix{ V \ar[r]^-{\mu_{V}}\ar[dr]_{f}   & {\mathcal{P}(V)}\ar[d]^{g} \\
{} & {P}
}
\end{eqnarray*}

A commutative bialgebra $B$ together with a Poisson bracket
$[\cdot, \, \cdot]_{B}$ is called a Poisson bialgebra if
the comultiplication $\Delta_{B}$ and the counit $\varepsilon_{B}$
are Poisson algebra homomorphisms.
Furthermore, if $B$ is a Hopf algebra then $B$ is called a Poisson Hopf algebra. It is straightforward to see that the
antipode $S_{B}$ is a Poisson algebra anti-morphism, i.e. for all
$a$, $b \in B$ we have: $S_{B}\bigl([a, \, b]_{B}\bigl) =
[S_{B}(b), \, S_{B}(a)]_{B}$. A morphism of Poisson
bialgebras is both a morphism of Poisson algebras and a morphism
of coalgebras. It can be easily seen that $\mathbf{PoissHopf}_F$ is a full subcategory of the category
$\mathbf{PoissBialg}_F$, i.e. a morphism of Poisson bialgebras between two
Poisson Hopf algebras is automatically a Poisson Hopf morphism.

Recall that both forgetful functors $\mathbf{PoissHopf}_F \to \mathbf{Coalg}_F$ and $\mathbf{PoissHopf}_F \to \mathbf{PoissBialg}_F$ admit left adjoints; in other words, there exists a free Poisson Hopf algebra on every Poisson bialgebra (resp. coalgebra). In particular, we have the following (\cite[Theorem 4.2]{A1}):

\begin{proposition}\label{00.00}
Let $B$ be a Poisson bialgebra. There exists a Poisson Hopf algebra $H(B)$ and a Poisson bialgebra homomorphism $\alpha_{B}\colon B \to H(B)$ such that for every Poisson Hopf algebra $P$ and every Poisson bialgebra homomorphism $f \colon B \to P$, there is a unique Poisson Hopf algebra homomorphism $g \colon H(B) \to P$ such that the following diagram is commutative:
\begin{eqnarray*}
\xymatrix{ B \ar[r]^-{\alpha_{B}}\ar[dr]_{f}   & {H(B)}\ar[d]^{g} \\
{} & {P}
}
\end{eqnarray*}
\end{proposition}

Let $H$ be a Poisson bialgebra/Hopf algebra. Then a Poisson $H$-comodule algebra (see \cite{BHO}, for example) is a Poisson algebra $P$ endowed with a comodule structure $\rho \colon P \to P \otimes H$ over $H$ which is a Poisson algebra homomorphism ($P \otimes H$ has the usual tensor product Poisson algebra structure given in \ref{1}),  i.e. for all $p$, $q \in P$ we have: 
\begin{eqnarray}
\rho\bigl([p,\,q]_{P}\bigl) = [\rho(p),\, \rho(q)]_{P \otimes Q}, \quad \rho(ab)= \rho(a) \rho(b), \quad\rho(1_P)=1_P\otimes 1_{H}.\label{1.2}
\end{eqnarray}
For unexplained notions pertaining to Hopf algebra theory we refer to the classical textbook \cite{Abe}
and to \cite{LPV} for basic properties of Poisson algebras.

\section{The universal coacting Poisson Hopf algebra of a Poisson algebra}\label{univ_coact}

In this section we introduce the universal coacting Poisson Hopf algebra of a Poisson algebra. To start with, we first construct the universal Poisson algebra of a pair of Poisson algebras. 

\begin{definition}\label{00}
Let $P$, $U$ be two Poisson algebras. The \textit{universal Poisson algebra of $P$ and $U$} is a Poisson algebra $\mathcal{B}(P, U)$ together with a Poisson algebra homomorphism $\psi_{\mathcal{B}(P, U)}\colon P \to U \otimes\, \mathcal{B}(P, U)$ such that for any Poisson algebra $Q$ and any Poisson algebra homomorphism $f \colon P \to U \otimes Q$ there exists a unique Poisson algebra homomorphism $g\colon \mathcal{B}(P, U) \to Q$ which makes the following diagram commutative:
\begin{equation}\label{0}
\xymatrix{ P\ar[rr]^-{\psi_{\mathcal{B}(P, U)}}\ar[rrd]_{f}  & {} & {U \otimes \mathcal{B}(P, U)}\ar[d]^{ \mathds{1}_{U} \otimes g}  \\
{} & {} & {U \otimes Q}
} \qquad {\rm i.e.}\,\,\, ( \mathds{1}_{U} \otimes g) \circ \psi_{\mathcal{B}(P, U)} = f.
\end{equation}
If $P=U$ we denote the corresponding universal Poisson algebra $\mathcal{B}(P, P)$ simply by $\mathcal{B}(P)$. 
\end{definition}

\begin{theorem}\label{t1}
If $P$ and $U$ are two Poisson algebras and $U$ is finite dimensional, there exists a universal Poisson algebra of $P$ and $U$.
\end{theorem}
\begin{proof}
Consider $B_{P} = \{x_{i} ~|~ i \in I \}$ and $B_{U} = \{y_{j} ~|~ j = 1, 2, \cdots, n\}$ to be $F$ bases of $P$ and respectively $U$ such that $x_{1} = 1_{P}$ and $y_{1} = 1_{U}$. For all $i$, $j \in I$ there exist two finite subsets $A_{ij}$ and $B_{ij}$ of $I$ such that: 
\begin{eqnarray}
x_{i}\,x_{j} = \sum_{k \in A_{ij}} \alpha_{ij}^{k} \, x_{k}, && \big[x_{i},\,x_{j}\big]_{P} = \sum_{k \in B_{ij}} \beta_{ij}^{k} \, x_{k},\label{eq1}
\end{eqnarray}
for some scalars $\alpha_{ij}^{k}$, $\beta_{ij}^{k}$ in $F$. Similarly, for all $l$, $t$, $s = 1, 2, \cdots, n$ we denote by $\gamma_{lt}^{s}$, $\tau_{lt}^{s}$ the scalars in $F$ such that:
\begin{eqnarray}
y_{l}\,y_{t} = \sum_{s=1}^{n} \gamma_{lt}^{s} \, y_{s}, && \big[y_{l},\,y_{t}\big]_{U} = \sum_{s=1}^{n} \tau_{lt}^{s} \, y_{s}. \label{eq2}
\end{eqnarray}

Furthermore, consider $W$ to be the $F$-linear span of the formal variables $h_{si}$, where $i \in I$, $s = 1,2, \cdots, n$ and let $\bigl(\mathcal{P}(W),\, \mu_{W} \bigl)$ be the free Poisson algebra on $W$. We will show that the pair $\bigl(\mathcal{B}(P, U) = \mathcal{P}(W) / \mathcal{I}, \psi_{\mathcal{B}(P, U)}\bigl)$ is the universal Poisson algebra of $P$, where $\mathcal{I}$ is the Poisson ideal generated by the set $I = I_{1} \cup I_{2}$ with:
\begin{eqnarray*}
&&\hspace*{-10mm} I_{1} = \left \{\mu_{W}(h_{s1}) - \delta_{s,\,1} 1,\, \sum_{k \in A_{ij}} \alpha_{ij}^{k}\, \mu_{W}(h_{si}) - \sum_{l, t =1}^{n} \gamma_{l t}^{s} \, \mu_{W}(h_{li}) \,\mu_{W}(h_{tj})~ \middle | ~ i, j \in I, s = 1, 2, \cdots, n  \right \}\\
&&\hspace*{-10mm} I_{2} = \left \{\sum_{k \in B_{ij}} \beta_{ij}^{k} \,\mu_{W}(h_{sk}) - \sum_{u, v=1}^{n} \, \Bigl(\gamma_{uv}^{s}\, \big[\mu_{W}(h_{ui}),\, \mu_{W}(h_{vj})\big]_{\mathcal{P}(W)} + \tau_{uv}^{s}\, \mu_{W}(h_{ui})\, \mu_{W}(h_{vj})\Bigl) ~ \middle | ~  \right.\\
&& \left. i, j \in I, s = 1, 2, \cdots, n \vphantom{\sum_{k=1}^{n}} \right\}
\end{eqnarray*}
and the Poisson algebra homomorphism $\psi_{\mathcal{B}(P, U)}$ is defined as follows for all $i \in I$:
\begin{equation}\label{eq2}
\psi_{\mathcal{B}(P, U)}(x_{i}) = \sum_{s=1}^{n} y_{s} \otimes \Bigl(\pi \circ \mu_{W}(h_{si})\Bigl)
\end{equation} 
where $\pi \colon \mathcal{P}(W) \to  \mathcal{B}(P, U)$ is the residue class homomorphism.
To this end, consider $f \colon P \to U \otimes Q$ to be a Poisson algebra homomorphism and denote 
$$
f(x_{j}) = \sum_{s=1}^{n} \, y_{s} \otimes d_{sj}, \,\,\, {\rm for\,\, all}\,\, j \in I,
$$
where $d_{sj} \in Q$. Moreover, as $f$ is a Poisson algebra homomorphism we also have:
\begin{eqnarray}
\hspace{0.7cm} && d_{s1} = \delta_{s,\,1} 1,\, \sum_{k \in A_{ij}} \alpha_{ij}^{k}\, d_{si} = \sum_{l, t =1}^{n} \gamma_{l t}^{s} \, d_{li} \,d_{tj}\\
&& \sum_{k \in B_{ij}} \beta_{ij}^{k} \,d_{sk} = \sum_{u, v=1}^{n} \, \Big(\gamma_{uv}^{s}\, \big[d_{ui},\, d_{vj}\big]_{Q} + \tau_{uv}^{s}\, d_{ui}\, d_{vj}\Big)
\end{eqnarray}
for all $i$, $j \in I$, $s = 1, 2, \cdots, n$.

Now the linear homomorphism $u \colon W \to Q$ defined by $u(h_{si}) = d_{si}$ for all $s = 1, 2, \cdots, n$, $i \in I$, extends uniquely to a Poisson algebra homomorphism $v \colon \mathcal{P}(W) \to Q$ such that $v \circ \mu_{W} = u$. Moreover, it can be easily seen that $I \subseteq {\rm ker} \,v$ and therefore there exists a unique Poisson algebra homomorphism $g \colon \mathcal{B}(P, U) \to Q$ such that $g \circ \pi = v$. Putting everything together yields:
\begin{eqnarray*}
(\mathds{1}_{U} \otimes g) \circ \psi_{\mathcal{B}(P, U)}(x_{i}) &=& (\mathds{1}_{U} \otimes g) \Bigl(\sum _{l=1}^{n} y_{l} \otimes \bigl(\pi\circ \mu_{W}(h_{li})\bigl) \Bigl)\\
 &=& \sum _{l=1}^{n} y_{l} \otimes \bigl(v \circ \mu_{W}(h_{li})\bigl)\\
 &=& \sum _{l=1}^{n} y_{l} \otimes u(h_{li})\\
 &=& \sum _{l=1}^{n} y_{l} \otimes d_{li} = f(x_{i})
\end{eqnarray*}
for all $i \in I$. Hence $g$ is the unique Poisson algebra homomorphism which makes diagram \ref{00} commute and the proof is now finished.
\end{proof}

\begin{corollary}\label{corbij}
Let $P$ and $U$ be two Poisson algebras with $U$ finite dimensional. For any Poisson algebra $Q$ we have a bijective correspondence between:
\begin{enumerate}
\item[(1)] Poisson algebra homomorphisms $f \colon P \to U \otimes Q$;
\item[(2)] Poisson algebra homomorphisms $g \colon \mathcal{B}(P, U) \to Q$.
\end{enumerate} 
\end{corollary}

The universal Poisson algebra gives rise to two functors defined as follows:

\begin{proposition}\label{pro}
Let $P$ and $U$ be two given Poisson algebras with $U$ finite dimensional. 
\begin{enumerate}
\item[1)] There exists a functor $\mathcal{L}(-,\, U) \colon \mathbf{Poiss}_{F} \to \mathbf{Poiss}_F$ defined as follows for any Poisson algebras $X$, $Y$ and any morphism $f \colon X \to Y$ in $\mathbf{Poiss}_F$:
$$
\mathcal{L}(X,\, U) = \mathcal{B}(X, U),\,\,\, \mathcal{L}(f,\, U)= \overline{f}  
$$
where $\overline{f} $ is the unique Poisson algebra homomorphism which makes the following diagram commute:
\begin{eqnarray}\label{funct1}
\xymatrix{ X\ar[rr]^-{\psi_{\mathcal{B}(X, U)}}\ar[rrd]_-{\psi_{\mathcal{B}(Y, U)} \circ f}  & {} & {U \otimes \mathcal{B}(X, U)}\ar[d]^{ \mathds{1}_{U} \otimes \overline{f}}  \\
{} & {} & {U \otimes \mathcal{B}(Y, U)}
} \qquad {\rm i.e.}\,\,\, (\mathds{1}_{U} \otimes \overline{f} ) \circ \psi_{\mathcal{B}(X, U)} = \psi_{\mathcal{B}(Y, U)} \circ f.
\end{eqnarray}

\item[2)] There exists a functor $\mathcal{R}(P,\, -) \colon \bigl(\mathbf{Poiss}_{F}^{fd}\bigr)^{op} \to \mathbf{Poiss}_F$ defined as follows for any finite dimensional Poisson algebras $A$, $B$ and any morphism $f^{op} \colon A \to B$ in $\bigl(\mathbf{Poiss}_{F}^{fd}\bigr)^{op}$:
$$
\mathcal{R}(P,\, A) = \mathcal{B}(P, A),\,\,\, \mathcal{R}(P,\, f^{op}) = \overline{f}  
$$
where $\overline{f} $ is the unique Poisson algebra homomorphism which makes the following diagram commute:
\begin{eqnarray*}
\xymatrix{ P\ar[rr]^-{\psi_{\mathcal{B}(P, A)}}\ar[rrdd]_-{(f \otimes \mathds{1}_{ \mathcal{B}(P, B)} )\circ \psi_{\mathcal{B}(P, B)} }  & {} & {A \otimes \mathcal{B}(P, A)}\ar[dd]^{ \mathds{1}_{A} \otimes \overline{f}}  \\
{} & {} & {}\\
{} & {} & {A \otimes \mathcal{B}(P, B)}
} \quad {\rm i.e.}\, (\mathds{1}_{A} \otimes \overline{f} ) \circ \psi_{\mathcal{B}(P, A)} = (f \otimes \mathds{1}_{ \mathcal{B}(P, B)} )\circ \psi_{\mathcal{B}(P, B)}. 
\end{eqnarray*}
\end{enumerate}
\end{proposition}
\begin{proof}
We only prove the first assertion. To start with, if $f = \mathds{1}_{X}$ is the identity morphism on a Poisson algebra $X$ then equation \ref{funct1} comes down to the identity $( \mathds{1}_{U} \otimes \overline{f}) \circ \psi_{\mathcal{B}(X, U)} = \psi_{\mathcal{B}(X, U)}$ and $\overline{f} = \mathds{1}_{\mathcal{B}(X, U)}$ is obviously  the unique Poisson algebra homomorphisms which fulfils it. Hence $\mathcal{L}(\mathds{1}_{X},\,U) = \mathds{1}_{\mathcal{B}(X, U)}$.  

Consider now $f \colon X \to Y$ and $g \colon Y \to Z$ two Poisson algebra homomorphisms and let $\overline{f} \colon \mathcal{B}(X, U) \to \mathcal{B}(Y, U)$ and respectively $\overline{g} \colon \mathcal{B}(Y, U) \to \mathcal{B}(Z, U)$ the unique Poisson algebra homomorphisms such that:
\begin{eqnarray}
(\mathds{1}_{U} \otimes \overline{f} ) \circ \psi_{\mathcal{B}(X, U)} &=& \psi_{\mathcal{B}(Y, U)} \circ f \label{funct2}\\
(\mathds{1}_{U} \otimes \overline{g} ) \circ \psi_{\mathcal{B}(Y, U)} &=& \psi_{\mathcal{B}(Z, U)} \circ g\label{funct3}
\end{eqnarray}
We are left to show that $\overline{g} \circ  \overline{f}$ is the unique Poisson algebra homomorphism which makes the following diagram commute:
$$
\xymatrix{ X\ar[rr]^-{\psi_{\mathcal{B}(X, U)}}\ar[rrdd]_-{\psi_{\mathcal{B}(Z, U)} \circ g \circ f}  & {} & {U \otimes \mathcal{B}(X, U)}\ar[dd]^{ \mathds{1}_{U} \otimes (\overline{g} \circ \overline{f})}\\
{} & {} & {} \\
{} & {} & {U \otimes \mathcal{B}(Z, U)}
} 
$$
Indeed, we have:
\begin{eqnarray*}
\bigl(\mathds{1}_{U} \otimes (\overline{g} \circ \overline{f}) \bigl)\circ \psi_{\mathcal{B}(X, U)} &=& (\mathds{1}_{U} \otimes \overline{g})\circ \underline{(\mathds{1}_{U} \otimes \overline{f})\circ \psi_{\mathcal{B}(X, U)}}\\
&\stackrel{\ref{funct2}} {=}& \underline{(\mathds{1}_{U} \otimes \overline{g})\circ \psi_{\mathcal{B}(Y, U)}} \circ f \\
&\stackrel{\ref{funct3}} {=}& \psi_{\mathcal{B}(Z, U)} \circ g \circ f
\end{eqnarray*}
which ends the proof.
\end{proof}

\begin{theorem}\label{th11}
For any finite dimensional Poisson algebra $U$, the functor $\mathcal{L}(-,\,U)  \colon \mathbf{Poiss}_{F} \to \mathbf{Poiss}_F$ is left adjoint to the tensor functor $U \otimes -  \colon \mathbf{Poiss}_{F} \to \mathbf{Poiss}_F$.
\end{theorem}
\begin{proof}
Consider $\theta \colon {\rm Hom}_{\mathbf{Poiss}_F} \bigl(\mathcal{L}(-, \, U),\, -\bigl) \to {\rm Hom}_{\mathbf{Poiss}_F} \bigl(-, \, U \otimes - \bigl) $
defined as follows for all Poisson algebras $X$ and $Q$:
$$
\theta_{X,\, Q}(g) = (\mathds{1}_{U} \otimes g ) \circ \psi_{\mathcal{B}(X, U)}.
$$
By Corollary~\ref{corbij}, $\theta_{X,\, Q} $ is a bijection of sets for all Poisson algebras $X$ and $Q$. The proof will be finished once we show that 
$\theta$ is natural in both variables. We only show the naturality in the first variable. To this end, let $f \in {\rm Hom}_{\mathbf{Poiss}_F}(X',\,X)$. Then, for any $t \in {\rm Hom}_{\mathbf{Poiss}_F} \bigl(\mathcal{L}(X, \, U),\, Q\bigl)$ we have: 
\begin{eqnarray*}
{\rm Hom}_{\mathbf{Poiss}_F} \bigl(f,\, U \otimes Q\bigl) \circ\, \theta_{X,\, Q}(t) &=& (\mathds{1}_{U} \otimes t) \circ \underline{\psi_{\mathcal{B}(X, U)} \circ f}\\
&\stackrel{(\ref{funct1})} {=}& (\mathds{1}_{U} \otimes t) \circ (\mathds{1}_{U} \otimes \mathcal{L}(f, \, U)) \circ \psi_{\mathcal{B}(X', U)}\\
&=& (\mathds{1}_{U} \otimes t \circ \mathcal{L}(f, \, U)) \circ \psi_{\mathcal{B}(X', U)}\\
&=& \theta_{X',\, Q} \bigl(t \circ \mathcal{L}(f, \, U)\bigl)\\
&=& \theta_{X',\, Q} \circ {\rm Hom}_{\mathbf{Poiss}_F} \bigl(\mathcal{L}(f, \, U),\, Q \bigl)(t)
\end{eqnarray*}
Hence ${\rm Hom}_{\mathbf{Poiss}_F} \bigl(f,\, U \otimes Q\bigl) \circ\, \theta_{X,\, Q} = \theta_{X',\, Q} \circ {\rm Hom}_{\mathbf{Poiss}_F} \bigl(\mathcal{L}(f, \, U),\, Q \bigl)$, as desired.
\end{proof}

\begin{corollary}
Let $U$ be a Poisson algebra such that the universal coacting Poisson algebra $\mathcal{P}(P, \, U)$ exists for all Poisson algebras $P$. Then $U$ is finite dimensional.
\end{corollary}
\begin{proof}
The existence of a universal coacting Poisson algebra $\mathcal{P}(P, \, U)$ for all Poisson algebras $P$ induces a functor $\mathcal{L}(-,\,U)  \colon \mathbf{Poiss}_{F} \to \mathbf{Poiss}_F$ as in Proposition~\ref{pro}, 1). Moreover, by Theorem~\ref{th11} the functor $\mathcal{L}(-,\,U)$ has a right adjoint given by $U \otimes -  \colon \mathbf{Poiss}_{F} \to \mathbf{Poiss}_F$. In particular, the tensor functor $U \otimes -$ preserves limits. Recall (see for instance \cite[Proposition 3.2]{A1}) that limits in $\mathbf{Poiss}_F$ are constructed as simply the limits in ${}_{F}{\mathcal {M}}$ of the underlying vector spaces. Now it is straightforward to see that the tensor product functor $U \otimes - $ preserves products in ${}_{F}{\mathcal {M}}$ if and only if $U$ is finite dimensional.
\end{proof}

\begin{example}\label{ex1}
Let $U$ be a finite dimensional Poisson algebra and $\bigl(P_{i}\bigl)_{i \in I}$ a family of Poisson algebras whose coproduct in $\mathbf{Poiss}_{F}$ we denote by $P$. Then, in light of Theorem~\ref{th11} we have an isomorphism of Poisson algebras:
$$\mathcal{B}(P, U) \cong \coprod_{i
\in I} \mathcal{B}\bigl(P_{i}, U\bigl)$$
where $\coprod_{i
\in I} \mathcal{B}\bigl(P_{i}, U\bigl)$ denotes the coproduct in $\mathbf{Poiss}_{F}$ of the universal coacting Poisson algebras $\mathcal{B}\bigl(P_{i}, U\bigl)$, $i \in I$.
We refer the reader to \cite{A1} for more detail on the (co)completeness of the categories $\mathbf{Poiss}_{F}$, $\mathbf{PoissBialg}_F$, $\mathbf{PoissHopf}_F$ and explicit constructions of certain (co)limits including coproducts.  
\end{example}

\begin{theorem}\label{imp1}
Let $P$ and $U$ be two given Poisson algebras with $U$ finite dimensional, $\bigl(\mathcal{B}(P, U),\, \psi_{\mathcal{B}(P, U)}\bigl)$ the corresponding universal coacting Poisson algebra and $f \colon U \to P$ a Poisson algebra homomorphism. Then:
\begin{enumerate}
\item There exists a coassociative Poisson algebra homomorphism $\Delta_{f} \colon \mathcal{B}(P,\, U) \to \mathcal{B}(P,\, U) \otimes \mathcal{B}(P,\, U)$;
\item If $g \colon P \to U$ is another Poisson algebra homomorphism such that:
\begin{eqnarray}
\psi_{\mathcal{B}(P, U)} &=& \psi_{\mathcal{B}(P, U)} \circ (f \circ g)\label{eq111}\\
\psi_{\mathcal{B}(P, U)} &=& \bigl((g \circ f) \otimes \mathds{1}_{\mathcal{B}(P,\, U)}\bigl) \circ\, \psi_{\mathcal{B}(P, U)}\label{eq222}
\end{eqnarray}
then there exists a Poisson algebra homomorphism $\varepsilon_{g} \colon \mathcal{B}(P,\, U) \to F$ such that $\bigl(\mathcal{B}(P,\, U), \Delta_{f},\, \varepsilon_{g}\bigl)$ becomes a Poisson bialgebra.
\end{enumerate}
\end{theorem}
\begin{proof}
(1) Using Theorem~\ref{t1} we obtain a unique Poisson algebra homomorphism $\Delta_{f} \colon \mathcal{B}(P,\, U) \to \mathcal{B}(P,\, U) \otimes \mathcal{B}(P,\, U)$ for which the following diagram commutes:
\begin{eqnarray*}
\xymatrix{ P\ar[rrrrr]^-{\psi_{\mathcal{B}(P,\, U)}}\ar[rrrrrd]_-{\bigl((\psi_{\mathcal{B}(P,\, U)} \circ f)\otimes \mathds{1}_{\mathcal{B}(P,\, U)}\bigl)\circ\, \psi_{\mathcal{B}(P,\, U)}}  & {} & {} & {} & {} & {U \otimes \mathcal{B}(P,\, U)}\ar[d]^{\mathds{1}_{U} \otimes \Delta_{f}}  \\
{} & {} & {} & {} & {} & {U \otimes \mathcal{B}(P,\, U)  \otimes \mathcal{B}(P,\, U) }}
\end{eqnarray*}
\begin{equation}\label{bialg111}
{\rm i.e.} \,\,\,(\mathds{1}_{U} \otimes \Delta_{f}) \circ \psi_{\mathcal{B}(P,\, U)} = \bigl((\psi_{\mathcal{B}(P,\, U)} \circ f) \otimes \mathds{1}_{\mathcal{B}(P,\, U)}\bigl)\circ\, \psi_{\mathcal{B}(P,\, U)}
\end{equation}
We show first that $\Delta_{f}$ is coassociative. To this end, by Theorem~\ref{t1} we have a unique Poisson algebra homomorphism $\theta\colon \mathcal{B}(P,\, U) \to \mathcal{B}(P,\, U) \otimes \mathcal{B}(P,\, U) \otimes \mathcal{B}(P,\, U)$ such that the following diagram is commutative: 
$$
\xymatrix{P \ar[rr]^-{\psi_{\mathcal{B}(P,\, U)}}\ar[rrddd]_{\Bigl(\mathds{1}_{U}  \otimes \bigl((\Delta_{f} \otimes \mathds{1}_{\mathcal{B}(P,\, U)})\circ \Delta_{f}\bigl)\Bigl) \circ \psi_{\mathcal{B}(P,\, U)}}   &{}& U \otimes \mathcal{B}(P,\, U)\ar[ddd]^{\mathds{1}_{U} \otimes  \theta}\\
{} & {} & {}\\
{} & {} & {}\\
{} & {} & {U \otimes \mathcal{B}(P,\, U) \otimes \mathcal{B}(P,\, U) \otimes \mathcal{B}(P,\, U)}
} 
$$
\begin{equation}\label{105}
{\rm i.e.}\,\,\, (\mathds{1}_{U} \otimes \theta)  \circ \psi_{\mathcal{B}(P,\, U)} = \Bigl(\mathds{1}_{U}  \otimes \bigl((\Delta_{f} \otimes \mathds{1}_{\mathcal{B}(P,\, U)}) \circ \Delta_{f}\bigl)\Bigl) \circ\,  \psi_{\mathcal{B}(P,\, U)}.
\end{equation}
Obviously the Poisson algebra homomorphism $(\Delta_{f} \otimes \mathds{1}_{\mathcal{B}(P,\, U)}) \circ \Delta_{f}\colon \mathcal{B}(P,\, U) \to \mathcal{B}(P,\, U) \otimes \mathcal{B}(P,\, U) \otimes \mathcal{B}(P,\, U)$ makes the above diagram commute. We are left to prove that the Poisson algebra homomorphism $(\mathds{1}_{\mathcal{B}(P,\, U)}   \otimes \Delta_{f}) \circ \Delta_{f}$ makes the same diagram commutative. Indeed, we have:
\begin{eqnarray*}
&& \hspace*{-15mm}\Bigl(\mathds{1}_{U}  \otimes \bigl((\mathds{1}_{\mathcal{B}(P,\, U)} \otimes \Delta_{f}) \circ \Delta_{f}\bigl)\Bigl) \circ\, \psi_{\mathcal{B}(P,\, U)} = (\mathds{1}_{U}  \otimes \mathds{1}_{\mathcal{B}(P,\, U)}  \otimes \Delta_{f}) \circ \underline{(\mathds{1}_{U} \otimes \Delta_{f}) \circ \psi_{\mathcal{B}(P,\, U)}}\\
&\stackrel{(\ref{bialg111})} {=}& (\mathds{1}_{U}  \otimes \mathds{1}_{\mathcal{B}(P,\, U)}   \otimes \Delta_{f}) \circ \bigl((\psi_{\mathcal{B}(P,\, U)}\circ f) \otimes \mathds{1}_{\mathcal{B}(P,\, U)}\bigl) \circ\, \psi_{\mathcal{B}(P,\, U)}\\
&{=}& \bigl((\psi_{\mathcal{B}(P,\, U)}\circ f) \otimes \mathds{1}_{\mathcal{B}(P,\, U)} \otimes \mathds{1}_{\mathcal{B}(P,\, U)}\bigl) \circ \underline{(\mathds{1}_{U} \otimes \Delta_{f}) \circ\, \psi_{\mathcal{B}(P,\, U)}}\\
&\stackrel{(\ref{bialg111})} {=}& \bigl((\psi_{\mathcal{B}(P,\, U)}\circ f) \otimes \mathds{1}_{\mathcal{B}(P,\, U)} \otimes \mathds{1}_{\mathcal{B}(P,\, U)}\bigl) \circ \bigl((\psi_{\mathcal{B}(P,\, U)} \circ f) \otimes \mathds{1}_{\mathcal{B}(P,\, U)}\bigl)\circ\, \psi_{\mathcal{B}(P,\, U)}\\
&{=}&  \Bigl(\underline{\bigl((\psi_{\mathcal{B}(P,\, U)}\circ f) \otimes \mathds{1}_{\mathcal{B}(P,\, U)}\bigl) \circ \psi_{\mathcal{B}(P,\, U)}} \otimes \mathds{1}_{\mathcal{B}(P,\, U)}\Bigl) \circ (f \otimes \mathds{1}_{\mathcal{B}(P,\, U)})\circ\, \psi_{\mathcal{B}(P,\, U)}\\
&\stackrel{(\ref{bialg111})} {=}&  \Bigl((\mathds{1}_{U} \otimes \Delta_{f}) \circ \psi_{\mathcal{B}(P,\, U)}  \otimes \mathds{1}_{\mathcal{B}(P,\, U)}\Bigl) \circ (f \otimes \mathds{1}_{\mathcal{B}(P,\, U)})\circ\, \psi_{\mathcal{B}(P,\, U)}\\
&{=}& \bigl(\mathds{1}_{U} \otimes \Delta_{f}  \otimes \mathds{1}_{\mathcal{B}(P,\, U)}\bigl) \circ \underline{\bigl((\psi_{\mathcal{B}(P,\, U)} \circ f) \otimes \mathds{1}_{\mathcal{B}(P,\, U)}\bigl)\circ\, \psi_{\mathcal{B}(P,\, U)}}\\
&\stackrel{(\ref{bialg111})} {=}& \bigl(\mathds{1}_{U} \otimes \Delta_{f}  \otimes \mathds{1}_{\mathcal{B}(P,\, U)}\bigl) \circ (\mathds{1}_{U} \otimes \Delta_{f}) \circ \psi_{\mathcal{B}(P,\, U)}\\
&{=}& \Bigl(\mathds{1}_{U} \otimes \bigl((\Delta_{f}  \otimes \mathds{1}_{\mathcal{B}(P,\, U)}) \circ \Delta_{f} \bigl)\Bigl) \circ\, \psi_{\mathcal{B}(P,\, U)}
\end{eqnarray*}
and thus $\Delta_{F}$ is coassociative. 

(2) Let $\varepsilon_{g}\colon \mathcal{B}(P,\, U) \to F$ be the unique Poisson algebra homomorphism such that the following diagram is commutative:
\begin{equation}\label{001}
\xymatrix{ P\ar[rr]^-{\psi_{\mathcal{B}(P, U)}}\ar[rrd]_{\mu_{U} \circ g}  & {} & {U \otimes \mathcal{B}(P, U)}\ar[d]^{ \mathds{1}_{U} \otimes \varepsilon_{g}}  \\
{} & {} & {U \otimes F}
} \qquad {\rm i.e.}\,\,\, ( \mathds{1}_{U} \otimes \varepsilon_{g}) \circ \psi_{\mathcal{B}(P, U)} = \mu_{U} \circ g.
\end{equation}
The proof will be finished once we show that $(\mathds{1}_{\mathcal{B}(P,\, U)} \otimes \varepsilon_{g}) \circ \Delta_{f}  = \mu_{\mathcal{B}(P,\, U)}$ and $(\varepsilon_{g} \otimes \mathds{1}_{\mathcal{B}(P,\, U)}) \circ \Delta_{f}  = \tau \circ \mu_{\mathcal{B}(P,\, U)}$. It will be enough to prove that:
\begin{eqnarray}
\Bigl(\mathds{1}_{U} \otimes (\mathds{1}_{\mathcal{B}(P,\, U)} \otimes \varepsilon_{g}) \circ \Delta_{f}\Bigl) \circ \, \psi_{\mathcal{B}(P, U)} &=& \bigl(\mathds{1}_{U} \otimes \mu_{\mathcal{B}(P,\, U)}) \circ \, \psi_{\mathcal{B}(P, U)}\label{aa11}\\
\Bigl(\mathds{1}_{U} \otimes (\varepsilon_{g} \otimes \mathds{1}_{\mathcal{B}(P,\, U)} ) \circ \Delta_{f}\Bigl) \circ \, \psi_{\mathcal{B}(P, U)} &=& \bigl(\mathds{1}_{U} \otimes \tau \circ \mu_{\mathcal{B}(P,\, U)}) \circ \, \psi_{\mathcal{B}(P, U)}\label{aa12}
\end{eqnarray}
To this end, we have:
\begin{eqnarray*}
&& \hspace*{-15mm} \Bigl(\mathds{1}_{U} \otimes (\mathds{1}_{\mathcal{B}(P,\, U)} \otimes \varepsilon_{g}) \circ \Delta_{f}\Bigl) \circ \, \psi_{\mathcal{B}(P, U)} = \bigl(\mathds{1}_{U} \otimes \mathds{1}_{\mathcal{B}(P,\, U)} \otimes \varepsilon_{g}\bigl)\circ \underline{\bigl(\mathds{1}_{U} \otimes \Delta_{f}\bigl) \circ \, \psi_{\mathcal{B}(P, U)}}\\
&\stackrel{(\ref{bialg111})} {=}& \bigl(\mathds{1}_{U} \otimes \mathds{1}_{\mathcal{B}(P,\, U)} \otimes \varepsilon_{g}\bigl)\circ \bigl(\psi_{\mathcal{B}(P, U)} \circ f \otimes \mathds{1}_{\mathcal{B}(P,\, U)}\bigl) \circ \, \psi_{\mathcal{B}(P, U)}\\
&=& \bigl(\psi_{\mathcal{B}(P, U)} \circ f \otimes \mathds{1}_{F}\bigl) \circ \underline{\bigl(\mathds{1}_{U} \otimes \varepsilon_{g}\bigl)\circ \,  \psi_{\mathcal{B}(P, U)}}\\
&\stackrel{(\ref{001})} {=}&  \bigl(\psi_{\mathcal{B}(P, U)} \circ f \otimes \mathds{1}_{F}\bigl) \circ\, \mu_{U} \circ g\\
&=&  \bigl(\psi_{\mathcal{B}(P, U)} \circ f \otimes \mathds{1}_{F}\bigl) \circ\ (g \otimes 1_{F})\\
&=&  \underline{\psi_{\mathcal{B}(P, U)} \circ f  \circ g} \otimes 1_{F} \stackrel{(\ref{eq111})} {=} \psi_{\mathcal{B}(P, U)} \otimes 1_{F} = \bigl(\mathds{1}_{U} \otimes \mu_{\mathcal{B}(P,\, U)}\bigl) \circ \,\psi_{\mathcal{B}(P, U)} 
\end{eqnarray*}
Hence, \ref{aa11} holds. Furthermore, \ref{aa12} can be shown to hold true in a similar manner by using \ref{eq222}.
\end{proof}

\begin{definition}\label{01}
Let $P$ be a Poisson algebra. The \textit{universal coacting Poisson Hopf algebra of $P$} is a Poisson Hopf algebra $\mathcal{H}(P)$ together with a Poisson comodule algebra structure $\rho_{\mathcal{H}(P)}\colon P \to P \otimes \mathcal{H}(P)$ such that for any Poisson Hopf algebra $H$ and any Poisson comodule algebra structure $\rho_{H} \colon P \to P \otimes H$ there exists a unique Poisson Hopf algebra homomorphism $g\colon \mathcal{H}(P) \to H$ which makes the following diagram commutative:
\begin{equation}\label{0}
\xymatrix{ P\ar[r]^-{\rho_{\mathcal{H}(P)}}\ar[rd]_{\rho_{H}}  & {P \otimes \mathcal{H}(P)}\ar[d]^{ \mathds{1}_{P} \otimes g}  \\
{} & {P \otimes H}
} \qquad {\rm i.e.}\,\,\, ( \mathds{1}_{P} \otimes g) \circ \rho_{\mathcal{H}(P)} = \rho_{H}.
\end{equation}
\end{definition}
The corresponding definition of a universal coacting Poisson bialgebra is obtained by
replacing each occurrence of \textit{Poisson Hopf algebra} by \textit{Poisson bialgebra} in
the above definition.

\begin{theorem}\label{t2}
Let $P$ be a finite dimensional Poisson algebra. Then:
\begin{enumerate}
\item The universal coacting Poisson algebra $\mathcal{B}(P)$ of $P$ admits a unique coalgebra structure $(\Delta,\, \varepsilon)$ which makes $\psi_{\mathcal{B}(P)}$ a Poisson comodule algebra structure. Moreover, with this coalgebra structure $\bigl(\mathcal{B}(P),\, \Delta,\, \varepsilon \bigl)$ becomes a Poisson bialgebra and $\bigl(\mathcal{B}(P),\, \psi_{\mathcal{B}(P)}\bigl)$ is the universal coacting Poisson bialgebra of $P$;
\item The pair $\bigl(\mathcal{H}(P) = H(\mathcal{B}(P)), \, \rho_{\mathcal{H}(P)} =  (\mathds{1}_{P} \otimes \alpha_{\mathcal{B}(P)}) \circ \rho_{\mathcal{B}(P)} \bigl)$ is the universal coacting Poisson Hopf algebra of $P$.
\end{enumerate}
\end{theorem}
\begin{proof}
1) We apply Theorem~\ref{imp1} for $P=U$ and $f = g = \mathds{1}_{P}$. Hence, we have a unique Poisson algebra homomorphisms $\Delta \colon \mathcal{B}(P) \to \mathcal{B}(P) \otimes \mathcal{B}(P)$ for which diagram \ref{bialg111} commutes, i.e.:
\begin{equation}\label{bialg1}
(\mathds{1}_{P} \otimes \Delta) \circ \psi_{\mathcal{B}(P)} = (\psi_{\mathcal{B}(P)} \otimes \mathds{1}_{\mathcal{B}(P)})\circ \psi_{\mathcal{B}(P)}
\end{equation}
and a unique Poisson algebra homomorphism $\varepsilon \colon  \mathcal{B}(P) \to F$ such that diagram \ref{001} commutes, i.e.:
\begin{equation}\label{bialg2}
(\mathds{1}_{P} \otimes \varepsilon) \circ \psi_{\mathcal{B}(P)} = \mu_{P}. 
\end{equation}
As shown in the proof of Theorem~\ref{imp1}, $(\mathcal{B}(P), \Delta, \varepsilon)$ is a Poisson bialgebra. Furthermore, note that the compatibilities \ref{bialg1} and \ref{bialg2} imply that $\psi_{\mathcal{B}(P)}$ is also a comodule structure on $P$; hence $\psi_{\mathcal{B}(P)}$ is a Poisson $\mathcal{B}(P)$-comodule algebra structure on $P$.

We claim now that $\bigl(\mathcal{B}(P),\, \psi_{\mathcal{B}(P)}\bigl)$ is in fact the universal coacting Poisson bialgebra of $P$. Indeed, consider another Poisson bialgebra $B$ together with a Poisson comodule algebra structure $\rho_{B} \colon P \to P \otimes B$ on $P$. As $\bigl(\mathcal{B}(P),\, \psi_{\mathcal{B}(P)}\bigl)$ is the universal coacting Poisson algebra of $P$ we obtain a unique Poisson algebra homomorphism $g\colon \mathcal{B}(P) \to B$ which makes the following diagram commute:
\begin{eqnarray}\label{coalg1}
\xymatrix{ P\ar[r]^-{\psi_{\mathcal{B}(P)}}\ar[rd]_{\rho_{B}}  & {P \otimes \mathcal{B}(P)}\ar[d]^{ \mathds{1}_{P} \otimes g}  \\
{} & {P \otimes B}
} \qquad {\rm i.e.}\,\,\, ( \mathds{1}_{P} \otimes g) \circ \psi_{\mathcal{B}(P)} = \rho_{B}.
\end{eqnarray}
We are left to show that $g$ is also a coalgebra homomorphism. To this end, using again the fact that $\bigl(\mathcal{B}(P),\, \psi_{\mathcal{B}(P)}\bigl)$ is the universal coacting Poisson algebra of $P$, we obtain a unique Poisson algebra homomorphism $\tau \colon P \to P \otimes B \otimes B$ such that the following diagram is commutative:
\begin{equation}\label{101}
\xymatrix{ P\ar[r]^-{\psi_{\mathcal{B}(P)}}\ar[rdd]_{\bigl(\mathds{1}_{P} \otimes\, \Delta_{B} \circ g\bigl)\circ \psi_{\mathcal{B}(P)}}  & {P \otimes \mathcal{B}(P)}\ar[dd]^{ \mathds{1}_{P} \otimes \tau}  \\
{} & {} \\
{} & {P \otimes B \otimes B}
} \qquad {\rm i.e.}\,\,\, (\mathds{1}_{P} \otimes \tau) \circ \psi_{\mathcal{B}(P)} = \bigl(\mathds{1}_{P} \otimes\, \Delta_{B} \circ g\bigl)\circ \psi_{\mathcal{B}(P)}
\end{equation}  
Obviously $\Delta_{B} \circ g$ makes diagram~\ref{101} commute; the proof will be finished once we show that $(g \otimes g) \circ \Delta$ makes the same diagram commutative. Indeed, we have:
\begin{eqnarray*}
\bigl(\mathds{1}_{P} \otimes\, (g \otimes g) \circ \Delta\bigl)\circ \,\psi_{\mathcal{B}(P)} &=& \bigl(\mathds{1}_{P} \otimes g \otimes g \bigl)\circ \underline{\bigl(\mathds{1}_{P} \otimes \Delta\bigl)\circ \, \psi_{\mathcal{B}(P)}}\\
&\stackrel{(\ref{bialg1})} {=}& \bigl(\mathds{1}_{P} \otimes g \otimes g \bigl)\circ (\psi_{\mathcal{B}(P)} \otimes \mathds{1}_{\mathcal{B}(P)})\circ \psi_{\mathcal{B}(P)}\\
&=& \bigl(\underline{(\mathds{1}_{P} \otimes g) \circ  \psi_{\mathcal{B}(P)}} \otimes g\bigl)\circ \, \psi_{\mathcal{B}(P)}\\
&\stackrel{(\ref{coalg1})} {=}& \bigl(\rho_{B} \otimes g\bigl)\circ \, \psi_{\mathcal{B}(P)}\\
&=& (\rho_{B} \otimes  \mathds{1}_{B})\circ \underline{(\mathds{1}_{P} \otimes g) \circ \, \psi_{\mathcal{B}(P)}}\\
&\stackrel{(\ref{coalg1})} {=}& \underline{(\rho_{B} \otimes  \mathds{1}_{B})\circ \rho_{B}}\\
&=& (\mathds{1}_{P} \otimes \Delta_{B}) \circ \underline{\rho_{B}}\\
&\stackrel{(\ref{coalg1})} {=}& (\mathds{1}_{P} \otimes \Delta_{B}) \circ (\mathds{1}_{P} \otimes g) \circ \psi_{\mathcal{B}(P)}\\
&=& \bigl(\mathds{1}_{P} \otimes \Delta_{B} \circ g\bigl) \circ \,\psi_{\mathcal{B}(P)}
\end{eqnarray*}
as desired. Similarly it can be proved that $\varepsilon_{B} \circ g = \varepsilon$.

2) Recall that $\alpha_{\mathcal{B}(P)}$ is a Poisson bialgebra homomorphism and therefore $(\mathds{1}_{P} \otimes \alpha_{\mathcal{B}(P)}) \circ \rho_{\mathcal{B}(P)}$ is a Poisson comodule algebra structure on $P$. Consider now $H$ to be a Poisson Hopf algebra and $\rho_{H} \colon P \to P \otimes H$ a Poisson comodule algebra structure on $P$. Using the first part of the proof we obtain a unique Poisson bialgebra homomorphism $h \colon \mathcal{B}(P) \to H$ which renders the following diagram commutative:
\begin{eqnarray}\label{final1}
\xymatrix{ P\ar[r]^-{\rho_{\mathcal{B}(P)}}\ar[rd]_{\rho_{H}}  & {P \otimes \mathcal{B}(P)}\ar[d]^{ \mathds{1}_{P} \otimes h}  \\
{} & {P \otimes H}
} \qquad {\rm i.e.}\,\,\, ( \mathds{1}_{P} \otimes h) \circ \rho_{\mathcal{B}(P)} = \rho_{H}.
\end{eqnarray}
Now Proposition~\ref{00.00} yields a unique Poisson Hopf algebra homomorphism $g \colon H(\mathcal{B}(P)) \to H$ such that the following diagram is commutative:
\begin{eqnarray}\label{final2}
\xymatrix{ \mathcal{B}(P)\ar[r]^-{\alpha_{\mathcal{B}(P)}}\ar[rd]_{h}  & {H(\mathcal{B}(P))}\ar[d]^{ g}  \\
{} & {H}
} \qquad {\rm i.e.}\,\,\, g \circ \alpha_{\mathcal{B}(P)} = h.
\end{eqnarray}
We will show that $g$ is the unique Poisson Hopf algebra homomorphism which makes the following diagram commutative:
\begin{eqnarray*}
\xymatrix{ P\ar[rrrrr]^-{\bigl(\mathds{1}_{P} \otimes \alpha_{\mathcal{B}(P)}\bigl) \circ \rho_{\mathcal{B}(P)}}\ar[rrrrrd]_-{\rho_{H}} & {} & {} & {} & {} & {P \otimes H(\mathcal{B}(P))}\ar[d]^{\mathds{1}_{P} \otimes g}  \\
{} & {} & {} & {} & {} & {P \otimes H}}
\end{eqnarray*}
Indeed, we have:
\begin{eqnarray*}
(\mathds{1}_{P} \otimes g) \circ \bigl(\mathds{1}_{P} \otimes \alpha_{\mathcal{B}(P)}\bigl) \circ \rho_{\mathcal{B}(P)} &=&  \bigl(\mathds{1}_{P} \otimes \underline{g \circ \alpha_{\mathcal{B}(P)}}\bigl) \circ \rho_{\mathcal{B}(P)}\\
&\stackrel{(\ref{final2})} {=}& \underline{\bigl(\mathds{1}_{P} \otimes h \bigl) \circ \rho_{\mathcal{B}(P)}}\\
&\stackrel{(\ref{final1})} {=}& \rho_{H}.
\end{eqnarray*}
as desired. The fact that $g$ is the unique Poisson Hopf algebra homomorphism which renders the above diagram commutative follows from the uniqueness of the Poisson bialgebra homomorphism $h$ for which diagram~\ref{final1} is commutative. Therefore, $\bigl(H(\mathcal{B}(P)), \, (\mathds{1}_{P} \otimes \alpha_{\mathcal{B}(P)}) \circ \rho_{\mathcal{B}(P)} \bigl)$ is the universal coacting Poisson Hopf algebra of $P$ and the proof is now finished.
\end{proof}

\begin{corollary}
Let $P$ be a Poisson algebra. For any Poisson Hopf algebra $H$ we have a bijective correspondence between:
\begin{enumerate}
\item[(1)] Poisson comodule algebra structures $\rho_{H} \colon P \to P \otimes H$ on $P$;
\item[(2)] Poisson Hopf algebra homomorphisms $g \colon \mathcal{H}(P) \to H$.
\end{enumerate} 
\end{corollary}

\begin{example}
Let $P$ be a finite dimensional Poisson bialgebra and $\bigl(Q_{i}\bigl)_{i \in I}$ a family of Poisson bialgebras such that $Q_{i} = P$ for all $i \in I$. Then, using Example~\ref{ex1} we have an isomorphism between the following Poisson algebras: $$\mathcal{B}\bigl( \coprod_{i \in I} Q_{i},\,P\bigl) \cong \coprod_{i \in I} \mathcal{B}(P).$$
Now $ \mathcal{B}(P)$ is in fact a Poisson bialgebra by Theorem~\ref{t2} and moreover, \cite[Theorem 3.3]{A1} implies that $\coprod_{i \in I} \mathcal{B}(P)$ is also a Poisson bialgebra. Hence, $\mathcal{B}\bigl( \coprod_{i \in I} Q_{i},\,P\bigl)$ can be endowed with a Poisson bialgebra structure for any set $I$.
\end{example}

\end{document}